\documentclass[letterpaper, two column, 10 pt, conference]{ieeeconf}  

\IEEEoverridecommandlockouts                              
\usepackage[utf8]{inputenc}
\usepackage{amsmath,bm}
\usepackage{amsfonts}
\usepackage{amssymb}
\usepackage{graphicx}
\usepackage{amsfonts}
\usepackage{hyperref}
\usepackage{caption}
\usepackage{subcaption}
\usepackage{algorithm}
\usepackage{algpseudocode}
\usepackage{cases}
\usepackage{enumerate}
\usepackage{xcolor}
\usepackage{cite}
\usepackage[normalem]{ulem}
\usepackage{comment}

\newtheorem{Theorem}{Theorem}

\newtheorem{Lemma}{Lemma}

\newtheorem{Remark}{Remark}
\newtheorem{Assumption}{Assumption}
\newcommand{\m}[1]{\mathbf{#1}}
\newcommand{\mc}[1]{\mathcal{#1}}
\newcommand{\mb}[1]{\mathbb{#1}}
\newcommand{\abs}[1]{\lVert{#1} \rVert}
\newcommand\numberthis{\addtocounter{equation}{1}\tag{\theequation}}

\usepackage{tikz}
	\usetikzlibrary{shapes,arrows}
	\tikzstyle{frame} = [draw, -latex]
	\tikzstyle{line} = [draw]
	\tikzstyle{line2} = [draw, dashdotted]
	\tikzstyle{line3} = [draw, dashed]
	\tikzstyle{line3UD} = [draw, dashed]
	\tikzstyle{place} = [circle, draw=black, fill=white, thick, inner sep=2pt, minimum size=1mm]
	\tikzstyle{place2} = [circle, draw=black, fill=black, thick, inner sep=2pt, minimum size=1mm]
	\tikzstyle{placeRed} = [circle, draw=red, fill=red, thick, inner sep=2pt, minimum size=1mm]
	\tikzstyle{vertex} = [circle, draw=black, fill=black, thick, inner sep=2pt, minimum size=1mm]
\usepackage{tkz-euclide}
\usetikzlibrary{calc}

\newcommand{\RightAngle}[4][7pt]{%
        \draw ($#3!#1!#2$)
        --($ #3!2!($($#3!#1!#2$)!.5!($#3!#1!#4$)$) $)
        --($#3!#1!#4$) ;
        }

\title{\LARGE \bf
Discrete-Time Matrix-Weighted Consensus
}

\author{Quoc Van Tran, Minh Hoang  Trinh and Hyo-Sung Ahn
\thanks{This work was supported by  the National Research Foundation (NRF) of Korea under the grant NRF-2017R1A2B3007034.}
\thanks{Q. V. Tran and H.-S. Ahn are with the School of Mechanical Engineering,
        Gwangju Institute of Science and Technology, Gwangju 61005, Republic of Korea.
        E-mails: {\tt\small $\{$tranvanquoc, hyosung$\}$@gist.ac.kr}}
\thanks{M. H. Trinh is with the Department of Automatic Control, School of Electrical Engineering, Hanoi University of Science and Technology (HUST), Hanoi, Vietnam.
        E-mail: {\tt\small minh.trinhhoang@hust.edu.vn}}
}

\begin{document}

\maketitle
\thispagestyle{empty}
\pagestyle{empty}

\begin{abstract}
This article investigates discrete-time matrix-weighted consensus of multi-agent networks over undirected and connected graphs. We first present consensus protocols for the agents in common networks of symmetric matrix weights with possibly different update rates and switching network topologies. A special type of matrix-weighted consensus with non-symmetric matrix-weights that can render several consensus control scenarios such as ones with scaled/rotated updates and affine motion constraints is also considered. We employ Lyapunov stability theory for discrete-time systems and occasionally utilize Lipschitz continuity of the gradient of the Lyapunov function to show the convergence to a consensus of the agents in the system. Finally, simulation results are provided to illustrate the theoretical results.
\end{abstract}

\section{Introduction}
The task of reaching a consensus on local decision states of multiple agents in a system in a distributed fashion is a fundamental problem in many distributed algorithms over networked systems, such as coordination control \cite{Ren2007, Saber2004tac, Oh2015survey, Quoc2020tcns}, distributed optimization and machine learning \cite{KZhang2018icml, TYang2019, Quoc2020CDCextend}. In this context, each agent in the system holds a local (decision) state, which can be a scalar or a vector, and in order to reach a consensus each agent updates its state along the direction of a weighted sum of the relative states to its neighboring agents, which can be obtained via local inter-agent measurements or information exchanges.

Although consensus algorithms over scalar-weighted networks have been studied extensively, matrix-weighted consensus has been of particular interest recently. This is due to the fact that matrix-weights can capture inter-dependencies or impose cross-coupling constraints on the relative vectors of the agents, which is not achievable if only scalar weights are used. Therefore, systems with matrix-weights arise naturally in various disciplines of science and engineering including but not limited to matrix-weighted consensus/synchronization \cite{Tuna2016auto,Minh2018auto,LPan2019,Ahn2019auto}, opinion dynamics \cite{Friedkin2016Sci,Ahn2020socialSystems,Ben2020auto}, distributed control and estimation \cite{Barooah2008TSP,Zhao2018csm,Tuna2018tac}. In this line of research, our work in \cite{Minh2018auto} reveals that the existence of a positive spanning tree in the (undirected) graph of the system, i.e., a tree such that every edge weight is positive definite, is sufficient for the agents to achieve a consensus. In contrast, bipartite consensus can be achieved if the matrix-weighted graph is structurally balanced and contains a positive-negative spanning tree, whose edge weights are either positive or negative definite \cite{LPan2019}. There are also works in matrix-weighted consensus with time-varying network topologies \cite{LPan2020Arx} and directed graphs \cite{Chinnappa2019}.

The aforementioned works in matrix-weighted consensus have been investigated in continuous-time scenarios and often require that the interaction graph of the network has symmetric matrix-weights i.e., $\m{A}_{ij}=\m{A}_{ji}\geq 0$, where $\m{A}_{ij}$ and $\m{A}_{ji}$ are the matrix weights associated with two neighboring agents $i$ and $j$, respectively. However, discrete-time algorithms are relevant in discrete-time cyber-physical systems in which control and estimation algorithms are implemented in digital computers or micro-controllers. 
Furthermore, most of the existing distributed optimization and machine learning algorithms, and particularly, those based on (scalar-weighted) consensus, are in discrete-time setting \cite{TYang2019}. Therefore, in this work, we attempt to investigate discrete-time matrix-weighted consensus of multi-agent systems over undirected graphs.

The contributions of this paper are as follows. 
\begin{itemize}
\item We firstly study the discrete-time matrix-weighted consensus of multi-agent systems with undirected and connected graphs, in which the agents in the system can use different update rates or the same update rate. Furthermore, asymptotic convergence to the average consensus of the system with time-varying graph topologies is also established, provided that the union of the switching graphs over each successive time interval of the same length contains a positive spanning tree. The use of switching (matrix-weighted) graphs poses a mild assumption as it allows the network topology to be disconnected at any time instant and can be further utilized to reduce exchanged data per iteration significantly between two neighboring agents.
\item Secondly, consensus of the agents is examined when each agent $i$ in the system employs a same matrix weight $\m{A}_{ij}=\m{A}_{i}$ for every relative state to its neighbor $j$. Suppose that the interaction graph of the system is connected and the agents' update rates are sufficiently small. Then, when the matrix weight is (possibly non-symmetric) positive definite for all agents, which represent scaled consensus updates or small misalignments between the body-fixed coordinate systems of the agents, we show that the agents achieve a consensus. 
\item Thirdly, as an extension to the preceding case, we consider the case that the matrix weight $\m{A}_i$ associated with an agent $i$ can be positive semidefinite. We show that the state vector of agent $i$ is constrained in a linear subspace whose tangent space is spanned by the column space of $\m{A}_i$. Then, it is proven that if the intersection of the agents' subspaces is non-empty and the update rates of the agents are sufficiently small, the agents still achieve a consensus.
\item Finally, two simulation examples are provided to verify the theoretical development in the paper.
\end{itemize}

The rest of this paper is outlined as follows. Preliminaries and problem formulation are provided in Section \ref{sec:preliminary}. Section \ref{sec:unconstrained_consensus} presents consensus protocols for systems with symmetric matrix weights and possibly time-varying network topologies. The consensus over undirected networks with asymmetric matrix weights is investigated in Sections \ref{sec:asym_positive_weight} and \ref{sec:projection_weight}. Finally, simulation results are provided in Section \ref{sec:sim} and Section \ref{sec:conclusion} concludes this paper.

\section{Preliminaries and Problem Formulation}\label{sec:preliminary} 
\subsubsection*{Notation} Let $\mb{R}^d$ and $\mb{C}^d$ be the real and complex $d$-dimensional spaces, respectively. The set of nonnegative integers is $\mb{Z}^+$. The notation $||\cdot||$ denotes the Euclidean norm. Let $\mathrm{diag}(\m{A}_1,\ldots,\m{A}_n)\in \mb{R}^{N\times N},~N:=\sum_{i=1}^n d_i,$ be a block-diagonal matrix constructed from $\m{A}_1\in \mb{R}^{d_1\times d_1},\ldots,\m{A}_n\in \mb{R}^{d_n\times d_n}$. 
The Cartesian product of $\{\mc{X}_i\}_{i=1}^n\subseteq \mb{R}^d$ is denoted by $\prod_{i=1}^{n}\mc{X}_i$. The relation $\m{A}>0$ ($\m{A}\geq 0$) implies that the matrix $\m{A}$ is positive definite (positive semidefinite).
\subsection{Matrix Weighted Graph}
A \textit{matrix weighted graph} characterizing an interaction topology of a multi-agent network is denoted by $\mc{G}=(\mc{V},\mc{E},\mc{A})$, where, $\mc{V}=\{1,\ldots,n\}$ denotes the vertex set, $\mc{E}\subseteq\mc{V}\times \mc{V}$ denotes the set of edges of $\mc{G}$, and $\mc{A}=\{\m{A}_{ij}\in \mb{R}^{d\times d}:(i,j)\in \mc{E},\m{A}_{ij}\geq 0\}$. An edge is defined by the ordered pair $e_k=(i,j),i\neq j, k=1,\ldots,m,m=\vert \mathcal{E} \vert$. The graph $\mc{G}$ is said to be undirected if $(i,j)\in \mc{E}$ implies $(j,i)\in \mc{E}$, i.e. if $j$ is a neighbor of $i$, then $i$ is also a neighbor of $j$. If the graph $\mc{G}$ is directed, $(i,j)\in \mc{E}$ does not necessarily imply $(j,i)\in \mc{E}$. The set of neighboring agents of $i$ is denoted by $\mc{N}_i=\{j\in\mc{V}:(i,j)\in \mc{E}\}$. Associate with each edge $(i,j)\in \mc{E}$ the matrix weight $\m{A}_{ij}\geq 0$, and $\m{A}_{ij}= \m{0}$ when $(i,j)\not\in \mc{E}$. An edge $(i,j)$ is called to be \textit{positive definite} (\textit{positive semidefinite}) if $\m{A}_{ij}>0$ ($\m{A}_{ij}\geq0$). The \textit{matrix-weighted adjacency matrix} of $\mc{G}$ is given as $\m{A}=[\m{A}_{ij}]\in \mb{R}^{nd\times nd}$. 

We define $\m{D}_i:=\sum_{j\in \mc{N}_i}\m{A}_{ij}$ and let $\m{D}=\mathrm{diag}(\m{D}_1,\ldots,\m{D}_n)$ be the \textit{block-degree matrix} of the graph $\mc{G}$. Then, the \textit{matrix-weighted Laplacian} is given as $\m{L}=\m{D}-\m{A}\in \mb{R}^{nd\times nd}$. We denote $\m{L}^\mathrm{o}$ as the \textit{identity-matrix weighted Laplacian} of $\mc{G}$ with $\m{A}_{ij}=\m{I}_d$ for all $(i,j)\in \mc{E}$, and $\m{A}_{ij}= \m{0}$ otherwise.

When the matrix-weights in the graph are \textit{symmetric}, i.e., $\m{A}_{ij}=\m{A}_{ji}\geq 0,\forall (i,j)\in \mc{E}$, the following straightforwardly established lemma can be obtained \cite{Minh2018auto}.
\begin{Lemma} The matrix-weighted Laplacian $\m{L}$ is symmetric and positive semidefinite, and its null space is given as $\text{null}(\m{L})=\text{span}\{\text{range}(\m{1}_n\otimes \m{I}_d),\{\m{v}=[\m{v}_1^\top,\ldots,\m{v}_n^\top]^\top\in \mb{R}^{nd}:(\m{v}_j-\m{v}_i)\in \text{null}(\m{A}_{ij}),\forall (i,j)\in \mc{E}\}\}$.
\end{Lemma}
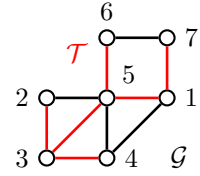
\begin{figure}[t]
\centering
\begin{tikzpicture}[scale=0.8]
\node[place] (4) at (1,0.) [label=right:$4$] {};
\node[place] (3) at (0,0) [label=left:$3$] {};
\node[place] (1) at (2,1.) [label=right:$1$] {};
\node[place] (2) at (0,1.) [label=left:$2$] {};
\node[place] (5) at (1.,1.) [label=above right:$5$] {};
\node[place] (6) at (1.,2) [label=above:$6$] {};
\node[place] (7) at (2,2) [label=right:$7$] {};
\node[] (g1) at (2.2,0.5) [label=below:$\mathcal{G}$] {};
\node[] (g1) at (0.5,2.2) [label=below: \textcolor{red}{$\mathcal{T}$}] {};

\draw (4) [line width=1pt] -- node [left] {} (5);
\draw (2) [line width=1pt] -- node [left] {} (5);
\draw (4) [line width=1pt] -- node [left] {} (1)  (6)--(7);

\draw [line width=1pt,red] (6) --(5) (5)--(3) (5)--(1) (1)--(7) (3)--(2) (3)--(4);

\end{tikzpicture}
\caption{A positive spanning tree $\mc{T}$ (in red) of $\mc{G}$ (\textcolor{red}{--} positive edges; -- positive or positive semidefinite edges).}
\label{fig:pos_spanning_tree}
\end{figure}
 A \textit{path} is positive if all the edges in the path are positive definite. A \textit{positive tree} is a graph in which any
two vertices are connected by exactly one path which is positive. A \textit{positive spanning tree} $\mc{T}$ of $\mc{G}$ is a positive tree containing all vertices in $\mc{V}$. When $\m{A}_{ij}$ and $\m{A}_{ji}$ are not necessarily equal, i.e., $\m{A}_{ij}\neq\m{A}_{ji}$, the graph $\mc{G}$ is said to have \textit{asymmetric matrix-weights}\footnote{The symmetry/asymmetry of the matrix weights of a graph $\mc{G}$, which is specified by whether $\m{A}_{ij}=\m{A}_{ji}, \forall (i,j)\in \mc{E}$, or not, should be distinguished from the symmetry of the positive semidefinite matrices $\m{A}_{ij}$.}.
\subsection{Problem Formulation}
Consider a system of $n$ agents in $\mb{R}^d,d\geq 2$, whose interaction graph $\mc{G}$ is undirected and connected. Each agent $i\in \mc{V}$ maintains a local vector $\m{x}_i\in \mb{R}^d$. Let each agent $i$ compute the relative vectors $(\m{x}_i-\m{x}_j)$ to its neighbors $j\in \mc{N}_i$, e.g., by assuming measurement capacity or by exchanging information with its neighbors. Intuitively, in order to reach a consensus, each agent $i$ in the network iteratively updates its local vector $\m{x}_i(k+1)$, at every iteration $k+1\geq 1$, by adding to it a matrix-weighted sum of the relative vectors, i.e., $-\sum_{j\in \mc{N}_i}\m{A}_{ij}(\m{x}_i(k)-\m{x}_j(k))$. Here, $\m{A}_{ij}\in \mb{R}^{d\times d}$ is a matrix weight associated with each edge $(i,j)\in \mc{E}$.

In particular, each agent $i\in \mc{V}$ can update $\m{x}_i(k+1)$ via
\begin{equation}\label{eq:unconstrained_update}
\m{x}_i(k+1)=\m{x}_i(k)-\alpha_i\sum_{j\in \mc{N}_i}\m{A}_{ij}(\m{x}_i(k)-\m{x}_j(k)),
\end{equation}
$\forall k\in \mb{Z}^+$, where $\alpha_i>0$ is a sufficiently small step size. In this work, we consider consensus control of the agents under two possible types of matrix weights $\m{A}_{ij}$:
\begin{enumerate}[({A}.1)]
\item Positive semidefinite and symmetric weights $\m{A}_{ij}=\m{A}_{ji}\geq 0, \forall (i,j)\in \mc{E}$ (Section \ref{sec:unconstrained_consensus}).
\item For every $i\in \mc{V}$, $\m{A}_{ij}=\m{A}_i\in \mb{R}^{d\times d}$ for all $j\in \mc{N}_i$. That is, each agent $i$ employs the same matrix weight $\m{A}_i$ for every relative vector $(\m{x}_i(k)-\m{x}_j(k)),\forall j\in \mc{N}_i$. In addition, the matrix $\m{A}_i$ is either positive definite (Section \ref{sec:asym_positive_weight}) or positive semidefinite (Section \ref{sec:projection_weight}) for all $i\in \mc{V}$, and it is also not required that $\m{A}_i=\m{A}_j$ for $i,j\in \mc{V}, i\neq j$.
\end{enumerate}
\section{Consensus under Symmetric Matrix-Weights}\label{sec:unconstrained_consensus}
In this section, we consider the consensus control for the system under the matrix-weighted consensus protocol \eqref{eq:unconstrained_update} with the matrix weights satisfying condition (A.1) above. Provided that the matrix-weighted graph $\mc{G}$ contains a positive spanning tree and the update rates are sufficiently small, we show that the agents achieve a consensus. Further, asymptotic convergence to an average consensus of the system under undirected switching graphs is also established.
\subsection{Matrix-weighted consensus law}
At an iteration $k=0,1,\ldots$, each agent $i$ updates its state vector $\m{x}_i(k)\in \mb{R}^d$ via \eqref{eq:unconstrained_update}. Let $\m{x}(k)=[\m{x}_1^\top(k),\ldots,\m{x}_n^\top(k)]^\top$ and $\m{G}=\mathrm{diag}\{\alpha_i^{-1}\m{I}_d\}_{i=1}^n$. Then, \eqref{eq:unconstrained_update} can be written in a more compact form
\begin{equation}\label{eq:unconstrained_update_matrix_form}
\m{x}(k+1)=(\m{I}_{dn}-\m{G}^{-1}\m{L})\m{x}(k).
\end{equation}
Select $\alpha_i=(||\m{D}_i||+\beta_i)^{-1}> 0,$ with $\beta_i >0$ being an arbitrary small constant, for all $i\in \mc{V}$. Since the matrix $\m{G}^{-1}\m{L}$ is non-symmetric, in order to study the stability of the system \eqref{eq:unconstrained_update_matrix_form}, we characterize the spectral property of the matrix $\m{I}_{dn}-\m{G}^{-1}\m{L}$ in what follows.
\begin{Lemma}\label{lm:sprectral_radius_unconstrained}
The matrix $(\m{I}_{dn}-\m{G}^{-1}\m{L})$ satisfies the following properties:
\begin{enumerate}[i)]
\item Its eigenvalues are real and its spectral radius is $\rho(\m{I}_{dn}-\m{G}^{-1}\m{L})=1$ with the corresponding eigenvectors are $\m{v}\in \text{null}(\m{L})$.
\item The unity eigenvalue $1$ of $(\m{I}_{dn}-\m{G}^{-1}\m{L})$ is \textit{semisimple}\footnote{An eigenvalue is semisimple if its algebraic multiplicity and geometric multiplicity are equal.}. As a result, $(\m{I}_{dn}-\m{G}^{-1}\m{L})$ is \textit{semi-convergent} or equivalently $\lim_{k\rightarrow \infty}(\m{I}_{dn}-\m{G}^{-1}\m{L})^k=(\m{I}_{dn}-\m{G}^{-1}\m{L})^\infty$ exists.
\end{enumerate}
\end{Lemma}
\begin{proof}
See Appendix \ref{app:sprectral_radius_unconstrained}.
\end{proof}
We next provide an explicit expression for the limit $\lim_{k\rightarrow \infty}(\m{I}_{dn}-\m{G}^{-1}\m{L})^k$. To proceed, consider the Jordan normal form of 
$$(\m{I}_{dn}-\m{G}^{-1}\m{L})=\m{VJV}^{-1},$$ 
where $\m{V}=[\m{v}_1,\ldots,\m{v}_{dn}]$ and $\m{V}^{-1}=[\m{u}_1,\ldots,\m{u}_{dn}]^\top$ denote matrices that contain the right eigenvectors and the left eigenvectors of $(\m{I}_{dn}-\m{G}^{-1}\m{L})$, respectively, in which the eigenvectors corresponding to the unity eigenvalues appear first. Let $\m{J}=\mathrm{diag}(1,\ldots,1,\m{J}_{l_2},\ldots,\m{J}_{l_p})\in \mb{R}^{nd\times nd}$ with the Jordan blocks $\m{J}_{l_i}\in \mb{R}^{l_i\times l_i},i=2,\ldots,p,\sum_{i=1}^pl_i=dn,$ corresponding to eigenvalues whose magnitudes are less than $1$. Then, we have
\begin{align*}
(\m{I}_{dn}-\m{G}^{-1}\m{L})^\infty&=\m{V}\m{J}^\infty\m{V}^{-1}\\
&=\m{V}\mathrm{diag}(1,\ldots,1,\m{J}_{l_2}^\infty,\ldots,\m{J}_{l_p}^\infty)\m{V}^{-1}\\
&=\m{V}\mathrm{diag}(1,\ldots,1,0,\ldots,0)\m{V}^{-1}\\
&=\textstyle\sum_{i=1}^{l_1}\m{v}_i\m{u}_i^\top, \numberthis \label{eq:infinity_matrix}
\end{align*}
where $l_1,d\leq l_1 < dn,$ is the number of unity eigenvalues of $(\m{I}_{dn}-\m{G}^{-1}\m{L})$, and $[\m{v}_1,\ldots,\m{v}_d]=\m{1}_n\otimes \m{I}_d$. From \eqref{eq:infinity_matrix}, we have $\lim_{k\rightarrow \infty}\m{x}(k)\in \mathrm{span}(\m{v}_1,\ldots,\m{v}_{l_1})$. Thus, the following theorem is obtained whose proof is given in Appendix \ref{app:unconstrained_convergence}.

\begin{Theorem}\label{thm:unconstrained_convergence} The sequence $\{\m{x}(k)\}$ generated by  \eqref{eq:unconstrained_update_matrix_form}, for an arbitrary initial vector $\m{x}(0)\in \mb{R}^{dn}$, converges geometrically to a consensus $\m{x}^*=\m{1}_n\otimes\hat{\m{x}},$ for a constant vector $\hat{\m{x}}\in \mb{R}^d$ if and only if $\text{null}(\m{L})=\text{range}(\m{1}_n\otimes \m{I}_d)$.
\end{Theorem}

\begin{Remark}
Note that though in \eqref{eq:unconstrained_update} the matrix weights are symmetric $\m{A}_{ij}=\m{A}_{ji}$, the agents employ different update rates $\alpha_i$. Thus, the agents are shown to achieve a consensus, but not necessarily the average consensus $\bar{\m{x}}:=\frac{1}{n}(\m{1}_n^\top\otimes \m{I}_d)\m{x}(0)$. In addition, the existence of a positive spanning tree in $\mc{G}$ is sufficient for the Laplacian matrix $\m{L}$ to satisfy the condition in Theorem \ref{thm:unconstrained_convergence} \cite{Minh2018auto}.
\end{Remark}
\subsection{Matrix-weighted average consensus}
Suppose that the agents use a common update rate $\alpha_i=\alpha=1/(\max_{i\in \mc{V}}( ||\m{D}_i||)+\beta) $, for all $i\in \mc{V}$. Such an update rate can be computed in a distributed manner using the max-consensus algorithm. 
 Then the iteration \eqref{eq:unconstrained_update_matrix_form} is rewritten as
\begin{equation}\label{eq:unconstrained_average_consensus}
\m{x}(k+1)=(\m{I}_{dn}-\alpha\m{L})\m{x}(k).
\end{equation}
It can be shown similarly as in Lemma \ref{lm:sprectral_radius_unconstrained} that $(\m{I}_{dn}-\alpha\m{L})$ has the spectral radius of one and is semi-convergent. In addition, the columns of $(\m{1}_n^\top\otimes \m{I}_d)$ are the left eigenvectors corresponding to the unity eigenvalues of $(\m{I}_{dn}-\alpha\m{L})$, i.e., $(\m{1}_n^\top\otimes \m{I}_d)(\m{I}_{dn}-\alpha\m{L})=(\m{1}_n^\top\otimes \m{I}_d)$. Then, we obtain the following theorem which can be proved by following similar lines as in Proof of Theorem \ref{thm:unconstrained_convergence}.
\begin{Theorem}\label{thm:unconstrained_average_convergence} The sequence $\{\m{x}(k)\}$ generated by  \eqref{eq:unconstrained_average_consensus}, for an arbitrary initial vector $\m{x}(0)\in \mb{R}^{dn}$, converges geometrically to the average consensus $\m{x}^*=\m{1}_n\otimes \bar{\m{x}}$ if and only if $\text{null}(\m{L})=\text{range}(\m{1}_n\otimes \m{I}_d)$.
\end{Theorem}
\subsection{Matrix-weighted average consensus under switching network topology}
The assumption on fixed interaction graphs can be relaxed by instead considering undirected switching graphs \cite{Jadbabaie2003tac,Saber2004tac,YSu2012auto}, which imposes a mild assumption on the interaction graphs and can reduce communicated data significantly in each iteration (see Remark \ref{rmk:compressed_data} below). To proceed, let $\sigma:\mb{Z}^+\rightarrow\mc{P}:=\{1,2,\ldots,\rho\}$ be a piecewise constant switching signal. That is, there exists a subsequence $k_l,l\in \mb{Z}^+$, of $\{k\},k\in \mb{Z}^+$, such that $\sigma(k)$ is a constant for $k_l\leq k< k_{l+1},\forall k_l$.

Given a switching signal $\sigma(k)$, we define an undirected switching graph $\mc{G}_{\sigma(k)}=\{\mc{V},\mc{E}_{\sigma(k)},\mc{A}_{\sigma(k)}\}$, where $\mc{V}=\{1,\ldots,n\}$ and $\mc{E}_{\sigma(k)}:=\{(i,j)\in \mc{V}\times \mc{V}:\m{A}_{ij}(k)=\m{A}_{ji}(k)\geq 0\}$. Note importantly that the condition $\m{A}_{ij}(k)=\m{A}_{ji}(k),\forall k\in \mb{Z}^+$, indicates that $\mc{G}_{\sigma(k)}$ remains undirected for every time instant $k$, but not necessarily connected. Let $\m{L}_{\sigma(k)}\in \mb{R}^{dn\times dn}$ be the corresponding matrix-weighted Laplacian of the graph $\mc{G}_{\sigma(k)}$. The union of such graphs $(\mc{G}_{\sigma(\gamma)},\mc{G}_{\sigma(\gamma+1)},\ldots,\mc{G}_{\sigma(\eta)})$ over a time interval $[\gamma,\eta]\subseteq [0,\infty)$, denoted as $\mc{G}_{\sigma(\gamma:\eta)}:=\cup_{k=\gamma}^{\eta}\mc{G}_{\sigma(k)}$, is defined by the triplet $\{\mc{V},\mc{E}_{\sigma(\gamma:\eta)},\mc{A}_{\sigma(\gamma:\eta)}\}$ . Here, the edge set $\mc{E}_{\sigma(\gamma:\eta)}:=\cup_{k=\gamma}^{\eta}\mc{E}_{\sigma(k)}$ and 
\begin{align*}
\mc{A}_{\sigma(\gamma:\eta)}:=\{\m{A}_{ij}(\gamma:\eta)&=\sum_{k=\gamma}^{\eta}\m{A}_{ij}(k):\\
&(i,j)\in \mc{E}_{\sigma(\gamma:\eta)}\}.
\end{align*} It is noted that $\m{A}_{ij}(\gamma:\eta)$ can be positive definite even if none of the weights $\{\m{A}_{ij}(k)\}_{k\in [\gamma,\eta]}$ is positive definite.
The graph $\mc{G}_{\sigma(k)}$ is assumed to satisfy the following joint connectedness assumption for matrix-weighted graphs \cite{Jadbabaie2003tac}.
\begin{Assumption}[Joint Connectedness]\label{ass:jointly_connected}
There exists a subsequence $\{k_t:t\in \mb{Z}^+\}$ such that $\lim_{t\rightarrow \infty}k_t=\infty$ and $k_{t+1}-k_t$ is uniformly bounded for all $t\geq 0$, and the graph $\cup_{k=k_t}^{k_{t+1}-1}\mc{G}_{\sigma(k)}$  contains a positive spanning tree.
\end{Assumption}
Joint connectedness of switching matrix-weighted graphs implies that the union of the switching graphs over each successive finite time span $[k_t,k_{t+1}-1]$ contains a positive spanning tree. Direct consequences of the joint connectedness of $\cup_{k=k_t}^{k_{t+1}-1}\mc{G}_{\sigma(k)}$ are as follows. The matrix-weighted Laplacian $\sum_{k=k_t}^{k_{t+1}-1}\m{L}_{\sigma(k)}$ of the graph $\cup_{k=k_t}^{k_{t+1}-1}\mc{G}_{\sigma(k)}$ is positive semi-definite, has $d$ zero eigenvalues and its null space is $\text{range}(\m{1}_n\otimes \m{I}_d)$.

\subsubsection*{Consensus Law} The consensus law for each agent $i\in \mc{V}$ under the switching graph $\mc{G}_{\sigma(k)}$ is given as
\begin{equation}\label{eq:consensus_switching_net}
\m{x}_i(k+1)=\m{x}_i(k)-\alpha\sum_{j=1}^n\m{A}_{ij}(k)(\m{x}_i(k)-\m{x}_j(k)),
\end{equation}
where $\alpha$ is a constant update rate to be defined, which is common to the agents.
The preceding consensus protocol can be written in a compact form
\begin{equation}\label{eq:consensus_switching_net_matrix}
\m{x}(k+1) = \m{x}(k)-\alpha\m{L}_{\sigma(k)}\m{x}(k).
\end{equation}
Let $\mu := \max_{\sigma(k)}||\m{L}_{\sigma(k)}||$. Then, we obtain the following theorem whose proof is given in Appendix \ref{app:switching_graph}.
\begin{Theorem}\label{thm:switching_graph}
Suppose that Assumption \ref{ass:jointly_connected} holds and the update rate $\alpha$ satisfies $0< \alpha < 1/\mu$. Then, the sequence $\{\m{x}(k)\}$ generated by  \eqref{eq:consensus_switching_net_matrix}, for an arbitrary initial vector $\m{x}(0)\in \mb{R}^{dn}$, asymptotically converges to the average consensus $\m{x}^*=\m{1}_n\otimes \bar{\m{x}}$ as $k\rightarrow \infty$.
\end{Theorem}
Theorem \ref{thm:switching_graph} indicates that joint connectedness condition on the switching graphs of the system is sufficient for the agents to achieve the average consensus, provided that the update rate is sufficiently small.
\begin{Remark}\label{rmk:compressed_data}
In the consensus of multi-agent systems whose state vectors are embedded in a high dimensional space $\mb{R}^d$, it is communication expensive for each agent $i$ to send the whole vector $\m{x}_i(k)$ to its neighbors at every iteration. We interpret here that how the matrix-weighted consensus law \eqref{eq:consensus_switching_net} can reduce the amount of exchanged data. Although the matrix weight $\m{A}_{ij}(k_t:k_{t+1}-1)=\sum_{k=k_t}^{k_{t+1}-1}\m{A}_{ij}(k)$ associated with an edge $(i,j)\in \mc{E}_{\sigma(k_t:k_{t+1}-1)}$ over a time interval $[k_t,k_{t+1}-1]$ might need to be positive definite (Assumption \ref{ass:jointly_connected}), $\m{A}_{ij}(k)$ can be positive semidefinite and even relatively low-rank, $\forall k\in \mb{Z}^+$. As a result, at each iteration $k\in \mb{Z}^+$, only a small portion of the coordinates of the vector $\m{x}_i(k)$ can be sent to agent $j$, and vice versa. For example, when $\m{A}_{ij}(k)=\mathrm{diag}(\m{B},\m{0})\in \mb{R}^{d\times d},$ for a matrix $\m{B}\in \mb{R}^{r\times r},\m{B}>0,r<d$, only the first $r$ components of $\m{A}_{ij}(k)\m{x}_i(k)$ need to be transmitted to agent $j$ since the other components are zeros. All the coordinates of the vector $\m{x}_i$ are evolved through interagent communications within each successive finite time span $k\in [k_t,k_{t+1}-1],\forall t\in \mb{Z}^+$.
Thus, the low-rank matrix weight $\m{A}_{ij}(k)$ acts as a compression operator that compresses a high-dimensional vector $\m{x}_i(k)$ before sending it at every iteration $k$.
\end{Remark}

\section{Consensus under Asymmetric Matrix-Weights}\label{sec:asym_positive_weight}
In this section, we suppose that each agent $i$ employs the same matrix weight $\m{A}_i$ for every relative vector $(\m{x}_i(k)-\m{x}_j(k)),\forall j\in \mc{N}_i$ (see condition (A.2) in Section \ref{sec:preliminary}). The matrix weight $\m{A}_i$ is assumed to satisfy Assumption \ref{ass:asym_matrix_weight} below, for all $i\in \mc{V}$. Under the connectedness condition on the graph $\mc{G}$ and sufficiently small update rates, we show that the system admits a consensus.
\subsection{Consensus Law}
Each agent $i$ updates its vector via
\begin{equation}\label{eq:asym_consensus_law}
\m{x}_i(k+1)=\m{x}_i(k)-\alpha_i\sum_{j\in \mc{N}_i}\m{A}_{i}(\m{x}_i(k)-\m{x}_j(k)),~\forall i\in \mc{V},
\end{equation}
where $\alpha_i>0$ is a step size (or update rate) associated with agent $i$, which is chosen sufficiently small to guarantee convergence of \eqref{eq:asym_consensus_law}. The matrix weight $\m{A}_{i}\in \mb{R}^{d\times d}$ associated with agent $i,\forall i \in \mc{V}$, is an invertible matrix, which is assumed to satisfy the following condition.
\begin{Assumption}\label{ass:asym_matrix_weight} There exists a positive constant $\gamma_i>0$ such that for any nonzero vector $\m{y}\in \mb{R}^d$, the following inequality holds
\begin{equation}\label{eq:asym_feasible_direction}
\m{y}^\top\m{A}_i^{-1}\m{y}\geq \gamma_i||\m{y}||^2.
\end{equation}
\end{Assumption}
Two possible classes of matrix weights that satisfy Assumption \ref{ass:asym_matrix_weight} are given as follows:
\begin{enumerate}[(i)]
\item Positive definite matrix weight $\m{A}_i>0$. Then \eqref{eq:asym_feasible_direction} is satisfied with $\gamma_i=\lambda^{-1}_{\max}(\m{A}_i)$.
\item Rotation matrices $\m{A}_i=\m{R}_i\in SO(d)$ that are (non-symmetric) positive definite, where $SO(d)$ denotes the \textit{special orthogonal group}. Indeed, using the relation $\m{R}_i^{-1}=\m{R}_i^\top$, for every nonzero vector $\m{y}\in  \mb{R}^d$, we have $$\m{y}^\top\m{R}_i^{-1}\m{y}=\m{y}^\top\m{R}_i^\top\m{y}=\m{y}^\top\m{R}_i\m{y}>0.$$ In addition, it follows from $\m{y}^\top(\m{R}_i\m{y})=\cos(\theta_i)||\m{y}||^2>0\Leftrightarrow \cos(\theta_i)\geq \gamma_i >0$, where $\theta_i$ is the angle between $\m{R}_i\m{y}$ and $\m{y}$, for a constant $\gamma_i\in (0,1)$. Consequently, $\m{y}^\top\m{R}_i^{-1}\m{y}\geq \gamma_i||\m{y}||^2$, which shows \eqref{eq:asym_feasible_direction}. 
\end{enumerate}
\begin{figure}[t]
\centering
\begin{tikzpicture}[scale=1]
\node[place] (j) at (4,0.) [label=below:$j$] {};
\node[place] (i) at (0,0.5) [label=below:$\m{x}_i(k)$] {};
\node[place] (k) at (4,1.) [label=above:$k$] {};

\draw[red,fill=yellow] (1,0.5) circle (0.5);
\draw[-latex,thick, blue](i)--(1.4,0.5) node[right] {$\m{u}_i(k)$};
\draw[-latex,thick, blue](i)--(1.,1.3) node[left,yshift=-1mm,xshift=-3mm] {$\m{A}_i\m{u}_i(k)$};
\node[place] () at (1.1,1.3) [label=right:$\m{x}_i(k+1)$] {};
\draw[-latex,thick, blue](j)--(3.4,0.2) ;
\draw[-latex,thick, blue](k)--(3.4,0.7) ;
\end{tikzpicture}
\caption{Interpretation of the matrix-weighted consensus scheme \eqref{eq:asym_consensus_law}. The desired displacement of consensus update $\m{u}_i(k)$ and the scaled/rotated update $\m{A}_i\m{u}_i(k)$ of agent $i$.}
\label{fig:collision_avoidance}
\end{figure}
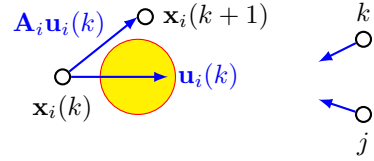
\begin{Remark}
The intuition of the consensus law \eqref{eq:asym_consensus_law} is as follows. Let $\m{u}_i(k):=-\alpha_i\sum_{j\in \mc{N}_i}(\m{x}_i(k)-\m{x}_j(k))$ be the gradient descent update direction of each agent $i$ that minimizes the objective function $V(\m{x})=(1/2) \m{x}^\top\m{L}^\mathrm{o}\m{x}=1/2\sum_{(i,j)\in \mc{E}}(\m{x}_i-\m{x}_j)^2$. Then, $\m{A}_i\m{u}_i(k)$ is the matrix-weighted consensus update of agent $i$ in \eqref{eq:asym_consensus_law} due to the scaled matrix/rotation $\m{A}_i$, as illustrated in Fig. \ref{fig:collision_avoidance}. 
Furthermore, the condition 
\begin{align*}
(\m{A}_i\m{u}_i)^\top\m{u}_i&=(\m{x}_i(k+1)-\m{x}_i(k))^\top\m{A}_i^{-1}\\
&\quad \times(\m{x}_i(k+1)-\m{x}_i(k))\\
&\overset{\eqref{eq:asym_feasible_direction}}{\geq} \gamma_i ||\m{x}_i(k+1)-\m{x}_i(k)||^2 \numberthis
\end{align*} indicates that $(\m{A}_i\m{u}_i)$ is indeed a descent direction, as it will be shown bellow that the function $V(\m{x})$ is non-increasing with respect to \eqref{eq:asym_consensus_law}.

The second case (ii) above also corresponds to the consensus of multiple agents in $\mb{R}^d$ in which the agent orientation matrices are measured with bias errors, if each agent is thought to maintain a body-fixed coordinate frame, whose origin is at its centroid, with regard to which the agent measure relative vectors.
Futhermore, in the case (ii), the consensus law \eqref{eq:asym_consensus_law} is a discrete-time counterpart of the continuous-time consensus law in \cite{Ahn2019auto}. As a development of \cite{Ahn2019auto}, the matrix-weighted consensus law \eqref{eq:asym_consensus_law} uses more general matrix weights and is applicable for an arbitrary $d$-dimensional space. 
\end{Remark}

\subsection{Convergence Analysis}
Let $\m{G}=\mathrm{diag}(\alpha_1\m{A}_1,\ldots,\alpha_n\m{A}_n)$, $\m{x}(k)=[\m{x}^{\top}_1(k),\ldots,\m{x}^{\top}_n(k)]^\top$. Then, Eq. \eqref{eq:asym_consensus_law} can be written as
\begin{equation}\label{eq:asym_consensus_matrix_form}
\m{x}(k+1)=\m{x}(k)-\m{G}\m{L}^\mathrm{o}\m{x}(k).
\end{equation}
Consider the Lyapunov function $V(\m{x}(k))=(1/2) \m{x}^\top(k)\m{L}^\mathrm{o}\m{x}(k)$, which is positive definite w.r.t. the consensus space $\mathrm{span}(\m{1}_n\otimes \m{I}_d)$. It is noted that the function $V(\m{x})$ is \textit{Lipschitz differentiable} with Lipschitz constant $L_{V}:=||\m{L}^\mathrm{o}||$, i.e., $\forall \m{x},\m{y}\in \mb{R}^d$, 
$$||\nabla V(\m{x})-\nabla V(\m{y})||=||\m{L}^\mathrm{o}(\m{x}-\m{y})||\leq ||\m{L}^\mathrm{o}||||\m{x}-\m{y}||.$$
An estimate of the upper-bound of the Laplacian spectral radius can be found in\cite{Liu2004LinearAlgeb}.
Let $\gamma_{\min}:=\min_{i=1,\ldots n} {\gamma_i}$ and $\alpha_{\max}:=\max_{i=1,\ldots n} {\alpha_i}$. Then, we have that the Lyapunov function $V(\m{x}(k))$ is non-increasing according to the following lemma.
\begin{Lemma}\label{lm:asym_non_increasing_function}
Suppose that the graph $\mc{G}$ is connected and Assumption \ref{ass:asym_matrix_weight} holds. Let the update rate $0<\alpha_{\max}<2\gamma_{\min}/L_{V}$. Then, the Lyapunov function $V(\m{x}(k))$ is non-increasing w.r.t. \eqref{eq:asym_consensus_matrix_form}, i.e.,
\begin{align*}
V(\m{x}(k+1))&-V(\m{x}(k))\leq -\gamma_{\min}\alpha_{\max}^{-1} ||\m{x}(k+1)-\m{x}(k)||^2. \numberthis \label{eq:asym_non_increasing_function}
\end{align*}
\end{Lemma}
\begin{proof} See Appendix \ref{app:asym_non_increasing_function}.
\end{proof}
From Lemma \ref{lm:asym_non_increasing_function}, convergence to a consensus of the system is shown in the following result.
\begin{Theorem}\label{thm:asym_pos_weight_consensus}
Suppose that the graph $\mc{G}$ is connected and Assumption \ref{ass:asym_matrix_weight} holds. If $0<\alpha_{\max}<2\gamma_{\min}/L_{V}$, the sequence $\{\m{x}(k)\}$ generated by \eqref{eq:asym_consensus_matrix_form}, for an arbitrary vector $\m{x}(0)\in \mb{R}^{dn}$, is bounded and converges geometrically to a consensus point $\lim_{k\rightarrow \infty}\m{x}(k)=\m{1}_n\otimes \m{x}^*$.
\end{Theorem} 
\begin{proof}
See Appendix \ref{app:asym_pos_weight_consensus}.
\end{proof}
\begin{Remark}
The result of Theorem \ref{thm:asym_pos_weight_consensus} further elaborates the robustness to orientation misalignments and flexibility of the consensus protocol \eqref{eq:asym_consensus_law} in modifying both the direction and magnitude of the displacement $\m{x}_i(k+1)-\m{x}_i(k)$ of each agent $i$ at each iteration $k$ (see Fig. \ref{fig:collision_avoidance}). Therefore, such flexible displacements can be utilized to design an obstacle avoidance scheme. For example, in Fig. \ref{fig:collision_avoidance}, agent $i$ changes its displacement to $\m{A}_i\m{u}_i$ in order to avoid collision with the obstacle (the yellow circle).
\end{Remark}
\section{Consensus under Asymmetric and Positive-Semidefinite Matrix Weights}\label{sec:projection_weight}
In this part, we consider the consensus scheme \eqref{eq:asym_consensus_law} under the scenario that the matrix weight $\m{A}_i$ associated with agent $i$ can be \textit{positive semi-definite}, $\forall i\in \mc{V}$. Therefore, $\m{A}_i$ is not necessarily invertible and consequently the convergence analysis in Section \ref{sec:asym_positive_weight} is not straightforwardly applicable for this case.
\subsection{Consensus Law}
We reuse the consensus law \eqref{eq:asym_consensus_law} below. In particular, each agent $i$ updates $\m{x}_i(k)$, for an initial vector $\m{x}_i(0)\in \mb{R}^d$, via
\begin{equation}\label{eq:constrained_consensus}
\m{x}_i(k+1)=\m{x}_i(k)-\alpha_i\sum_{j\in \mc{N}_i}\m{A}_{i}(\m{x}_i(k)-\m{x}_j(k)),\forall i\in \mc{V},
\end{equation}
where $\alpha_i> 0$ is a step size and $\m{A}_{i}\geq 0$ is a matrix weight, $\forall i \in \mc{V}$. 
We again use $\m{x}(k)=[\m{x}^{\top}_1(k),\ldots,\m{x}^{\top}_n(k)]^\top\in \mb{R}^{dn}$ to denote the stacked vector of all agent vectors. Let $\m{v}_i(k):=-\sum_{j\in \mc{N}_i}(\m{x}_i(k)-\m{x}_j(k))\in \mb{R}^d$ and hence \eqref{eq:constrained_consensus} can be written as
\begin{equation}\label{eq:asym_v_i}
\alpha_i\m{A}_i\m{v}_i(k)=\m{x}_i(k+1)-\m{x}_i(k).
\end{equation}
In the sequel, we show that the state vector of agent $i$, $\m{x}_i(k)$, is constrained in a linear subspace whose tangent space is spanned by the column space of $\m{A}_i$.
\subsection{Geometric Interpretation}
\begin{figure}[t]
\centering
\begin{tikzpicture}[scale=1]
\draw[fill=red,opacity=0.15] (-1,0,-1.5) -- (-1,0,1.5) -- (2.6,0,2.5) -- (2.6,0,-0.5) -- cycle;
\node[place, scale = 0.5] (xi) at (2.3,0,1.5)[label=right:$\m{x}_i(k)$]{};
\draw[-latex,thick, blue](xi)--(0.1,.,-2) node[right] {$\m{v}_i(k)$};
\draw[-latex,thick, blue](xi)--(0.8,.,0) node[left] {$\m{v}_i^t(k)$};
\draw[-latex,thick, red](0.8,.,0)--(0.1,.,-2) node[left] {$\m{v}_i^n(k)$};
 \RightAngle{(2.3,0,1.5)}{(0.8,.,0)}{(0.1,.,-2)};

\draw[-latex,thick](xi)--(0.7,0,1.3) node[below] {$\m{x}_i(k+1)$};
\node[place, scale = 0.5] at (0.7,0,1.3) []{};
\node[] at (-1.15,0,1.4)[label=right:$\mc{X}_i$]{};
\end{tikzpicture}
\caption{Geometric illustration of Proof of Lemma \ref{lm:constrained_feasible_direction}: the tangent component $\m{v}_i^t(k)=\m{P}_{\mc{TX}_i}\m{v}_i(k)$ and the normal component $\m{v}_i^n(k)=(\m{I}_d-\m{A}_i^\dag\m{A}_i)\m{v}_i(k)$. The normal vector satisfies $\m{v}_i^n(k)\perp \Delta \m{x}_i(k)$.}
\label{fig:asym_pos_semidef_weight}
\end{figure}
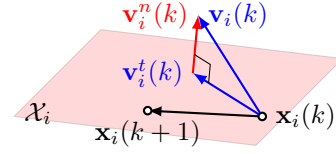
Since the matrix weight $\m{A}_i$ is positive semidefinite, we can decompose $\m{A}_i$ as $\m{A}_i=\m{V}_i\Sigma_i\m{V}_i^\top$, where $\Sigma_i:=\mathrm{diag}(\lambda_{i,1},\ldots,\lambda_{i,r_i},0,\ldots,0)\in \mb{R}^{d\times d}$ with $r_i$ ($1\leq r_i\leq d$) is being the rank of $\m{A}_i$ and $\lambda_{i,l}>0,l=1,\ldots, r_i,$ being the positive eigenvalues of $\m{A}_i$, and $\m{V}_i\in \mb{R}^{d\times d}$ is an orthogonal matrix. In addition, the first $r_i$ columns of $\m{V}_i$, i.e., $\m{V}_{i,1:r_i}:=[\m{v}_{i,1},\ldots,\m{v}_{i,r_i}]\in \mb{R}^{d\times r_i}$ form an  orthonormal basis of the range space of $\m{A}_i$. 

For each $i\in \mc{V}$ and given an initial vector $\m{x}_i(0)\in \mb{R}^d,$ we construct a (virtual) linear manifold (or subspace) $\mc{X}_i\subseteq \mb{R}^d$ such that $\m{x}_i(0)\in \mc{X}_i$ and the \textit{tangent space} of $\mc{X}_i$, denoted as $\mc{TX}_i$, satisfies $\text{span}(\mc{TX}_i)=\text{span}\{\m{v}_{i,1},\ldots,\m{v}_{i,r_i}\}$. It is noted that, given $\m{x}_i(0)\in \mb{R}^d$, such a subspace $\mc{X}_i$ is unique for every $i\in \mc{V}$. Furthermore, since $\text{range}(\m{A}_i)=\text{span}(\mc{TX}_i)$, it can be shown that  $\m{x}_i(k)\in \mc{X}_i$ for all time $k\in \mb{Z}^+$, $\forall i \in \mc{V}$. As a result, if the sequence $\{\m{x}(k)\}$ generated by \eqref{eq:constrained_consensus} converges to a consensus $\m{1}\otimes \m{x}^*$ for a point $\m{x}^*\in \mb{R}^d$ as $k\rightarrow \infty$, then the following condition must be satisfied.
\begin{Lemma}\label{lm:nonempty_intersection}
A necessary condition for the agents to achieve a consensus under the iterative update \eqref{eq:constrained_consensus} is the intersection of all manifolds $\mc{X}_i$ is non-empty $\mc{X}:=\cap_{i=1}^n\mc{X}_i\neq \emptyset$.
\end{Lemma}
Obviously, such a point $\m{x}^*\in \mc{X}$. In addition, the intersection set $\mc{X}$ is either a singleton or a linear subspace.
\begin{Remark}
The assumption above appears to be somewhat strict in the sense that it requires a careful selection of the matrix weights and initial vectors. Furthermore, from the condition of the non-empty intersection of some degenerate subspaces of $\mb{R}^{dn}$, it appears that for an arbitrary vector $\m{x}(0)\in \mb{R}^{dn}$, the agents almost surely do not reach a consensus. By way of contrast, it is shown in Section \ref{sec:asym_positive_weight} that the agents can always achieve a consensus under the update \eqref{eq:asym_consensus_law} with positive definite matrix-weights, assuming that the update rates are sufficiently small.
\end{Remark}

Define $\m{A}_i^\dag:=\m{V}_i\Sigma_i^\dag\m{V}_i^\top\in \mb{R}^{d\times d}$, where $\Sigma_i^\dag:=\mathrm{diag}(\lambda_{i,1}^{-1},\ldots,\lambda_{i,r_i}^{-1},0,\ldots,0)\in \mb{R}^{d\times d}$. Then, it can be shown that $\m{A}_i^\dag$ is the Moore-Penrose \textit{generalized inverse} of $\m{A}_i$, which satisfies: (i) $\m{A}_i\m{A}_i^\dag\m{A}_i=\m{A}_i$, (ii) $\m{A}_i^\dag\m{A}_i\m{A}_i^\dag=\m{A}_i^\dag$, and (iii) both $\m{A}_i\m{A}_i^\dag$ and $\m{A}_i^\dag\m{A}_i$ are symmetric \cite{Horn1985}. Moreover, the \textit{orthogonal projection matrix} that projects any vector onto the tangent space $\mc{TX}_i$ can be defined as
\begin{equation}\label{eq:asym_orthogonal_projection}
\m{P}_{\mc{TX}_i}:=\m{A}_i^\dag\m{A}_i=\m{V}_{i,1:r_i}\m{V}_{i,1:r_i}^\top.
\end{equation}
Note that the projection matrix $\m{P}_{\mc{TX}_i}$ is positive semidefinite, idempotent $\m{P}_{\mc{TX}_i}^2=\m{P}_{\mc{TX}_i}$, and contains $r_i$ unity eigenvalues and the other $(d-r_i)$ eigenvalues are zeros. 

\subsection{Convergence Analysis}
The following lemma is useful in showing the convergence of the system \eqref{eq:constrained_consensus}.
\begin{Lemma}\label{lm:constrained_feasible_direction}
Let $\Delta \m{x}_i(k):=\m{x}_i(k+1)-\m{x}_i(k)$ and $\m{v}_i(k)$ is defined above Eq. \eqref{eq:asym_v_i}. Then, for all $i\in \mc{V}$, the following inequality holds:
\begin{equation}\label{eq:constrained_feasible_direction}
\Delta \m{x}^\top_i(k)\m{v}_i(k)\geq \alpha_i^{-1}\lambda^{-1}_{\max}(\m{A}_i)||\Delta \m{x}_i(k)||^2.
\end{equation}
\end{Lemma}
\begin{proof}
First, by left-multiplying $\m{A}_i^\dag$ on both sides of \eqref{eq:asym_v_i}, one has
\begin{align*}
\alpha_i\m{A}_i^\dag\m{A}_i\m{v}_i(k)&=\m{A}_i^\dag\Delta \m{x}_i(k)\\
\Leftrightarrow \m{v}_i^t(k) &= \alpha_i^{-1}\m{A}_i^\dag\Delta \m{x}_i(k), \numberthis \label{eq:constrained_tangent_vi}
\end{align*}
where $\m{v}_i^t(k):=\m{A}_i^\dag\m{A}_i\m{v}_i(k)=\m{P}_{\mc{TX}_i}\m{v}_i(k)$ is the orthogonal projection of $\m{v}_i(k)$ onto the tangent space $\mc{TX}_i$, as illustrated in Fig. \ref{fig:asym_pos_semidef_weight}. Let $\m{v}_i^n(k):=\m{v}_i(k)-\m{v}_i^t(k)=(\m{I}_d-\m{A}_i^\dag\m{A}_i)\m{v}_i(k)$, which is orthogonal to the tangent space $\mc{TX}_i$ (or normal to the linear manifold $\mc{X}_i$). Then, consider the inner product
\begin{align*}
\Delta \m{x}^\top_i(k)\m{v}_i(k)&=\Delta \m{x}^\top_i(k)(\m{v}_i^t(k)+\m{v}_i^n(k))\\
&=\Delta \m{x}^\top_i(k)\m{v}_i^t(k)\\
&\overset{\eqref{eq:constrained_tangent_vi}}{=}\alpha_i^{-1}\Delta \m{x}^\top_i(k)\m{A}_i^\dag\Delta \m{x}_i(k),
\end{align*}
where the second equality follows from the relation $\m{v}_i^n(k)\perp \Delta \m{x}_i(k)$ (see also Fig. \ref{fig:asym_pos_semidef_weight}). Moreover, it is noted that $\m{A}_i^\dag\geq 0$ and from \eqref{eq:asym_v_i}, $\Delta \m{x}_i(k)\perp \text{null}(\m{A}_i)=\text{null}(\m{A}_i^\top)=\text{null}(\m{A}_i^\dag)$, for all $k\in \mb{Z}^+$. As a result, it follows from the preceding equation that
\begin{align*}
\Delta \m{x}^\top_i(k)\m{v}_i(k)&\geq \alpha_i^{-1}\lambda_{\min}(\m{A}_i^\dag)||\Delta \m{x}_i(k)||^2\\
&=\alpha_i^{-1}\lambda_{\max}^{-1}(\m{A}_i)||\Delta \m{x}_i(k)||^2,
\end{align*}
which completes the proof.
\end{proof}

To proceed, we define $\m{v}(k):=[\m{v}_1^\top(k),\ldots,\m{v}_n^\top(k)]^\top$ and consider the Lyapunov function $$V(\m{x}(k)):=\frac{1}{2} \m{x}^\top(k)\m{L}^\mathrm{o}\m{x}(k)=-\frac{1}{2} \m{x}^\top(k)\m{v}(k),$$ 
which is \textit{Lipschitz differentiable} with Lipschitz constant $L_{V}:=||\m{L}^\mathrm{o}||$. Furthermore, we let $\gamma_{\min}:=\min_{i\in \mc{V}} \lambda_{\max}^{-1}(\m{A}_i)$ and $\alpha_{\max}:=\max_{i\in \mc{V}} {\alpha_i}$.
Then, from the inequality \eqref{eq:constrained_feasible_direction} and by using a similar argument as in Proof of Lemma \ref{lm:asym_non_increasing_function}, we obtain the following result.
\begin{Lemma}\label{lm:constrained_non_increasing_function}
Suppose that the graph $\mc{G}$ is connected. Let the update rate $0<\alpha_i<2\gamma_{\min}/L_{V},\forall i=1,\ldots,n$. Then, the Lyapunov function $V(\m{x}(k))$ is non-increasing w.r.t. \eqref{eq:constrained_consensus}, i.e.,
\begin{align*}
0\leq V&(\m{x}(k+1))\leq\\
&V(\m{x}(k)) -\gamma_{\min}^{-1}\alpha_{\max} ||\m{x}(k+1)-\m{x}(k)||^2. \numberthis \label{eq:constrained_non_increasing_function}
\end{align*}
\end{Lemma}

\begin{Theorem}\label{thm:constrained_consensus}
Suppose that the graph $\mc{G}$ is connected and for $\m{x}(0)\in \mb{R}^{dn}$, the constructed linear manifolds have a non-empty intersection, $\mc{X}:=\cap_{i=1}^n\mc{X}_i\neq \emptyset$. Then, if $0<\alpha_{\max}<2\gamma_{\min}/L_V$, the sequence $\{\m{x}(k)\}$ generated by \eqref{eq:constrained_consensus} is bounded and converges geometrically to a consensus point $\lim_{k\rightarrow \infty}\m{x}(k)=\m{1}_n\otimes \m{x}^*$ for a point $\m{x}^*\in \mc{X}$.
\end{Theorem} 
\begin{proof} 
See Appendix \ref{app:constrained_consensus}.
\end{proof}
\section{Simulation}\label{sec:sim}

\subsection{Matrix-weighted consensus under switching graphs}
\begin{figure}[t]
\centering
\begin{subfigure}[b]{0.45\textwidth}
\centering
\begin{tikzpicture}[scale=0.7]
\node[place] (4) at (1,0.) [label=below:$4$] {};
\node[place] (1) at (0,0) [label=below:$1$] {};
\node[place] (2) at (0,1.) [label=above:$2$] {};
\node[place] (3) at (1.,1.) [label=above:$3$] {};
\node[] (g1) at (.5,-0.4) [label=below:$\mathcal{G}_1$] {};


\draw [line width=1pt,red] (1)--(2);
\node[place] (42) at (1+2.5,0.) [label=below:$4$] {};
\node[place] (12) at (+2.5,0) [label=below:$1$] {};
\node[place] (22) at (+2.5,1.) [label=above:$2$] {};
\node[place] (32) at (1.+2.5,1.) [label=above:$3$] {};
\node[] (g2) at (.5+2.5,-0.4) [label=below:$\mathcal{G}_2$] {};


\draw [line width=1pt,red] (12)--(42);

\node[place] (43) at (1+5,0.) [label=below:$4$] {};
\node[place] (13) at (+5,0) [label=below:$1$] {};
\node[place] (23) at (+5,1.) [label=above:$2$] {};
\node[place] (33) at (1.+5,1.) [label=above:$3$] {};
\node[] (g3) at (.5+5,-0.4) [label=below:$\mathcal{G}_3$] {};


\draw [line width=1pt] (23)--(33);

\node[place] (44) at (1+7.5,0.) [label=below:$4$] {};
\node[place] (14) at (+7.5,0) [label=below:$1$] {};
\node[place] (24) at (+7.5,1.) [label=above:$2$] {};
\node[place] (34) at (1.+7.5,1.) [label=above:$3$] {};
\node[] (g3) at (.5+7.5,-0.4) [label=below:$\mathcal{G}_4$] {};


\draw [line width=1pt] (24)--(34);

\end{tikzpicture}
\caption{Switching graphs $\mc{G}_{\sigma(k)}$ of the network with $\mc{P}=\{1,2,3,4\}$ (\textcolor{red}{--} positive edges; -- positive semi-definite edges).}
\label{fig:sim_switching_graphs}
\end{subfigure}\\
\begin{subfigure}[b]{0.45\textwidth}
\centering
\includegraphics[height=4cm]{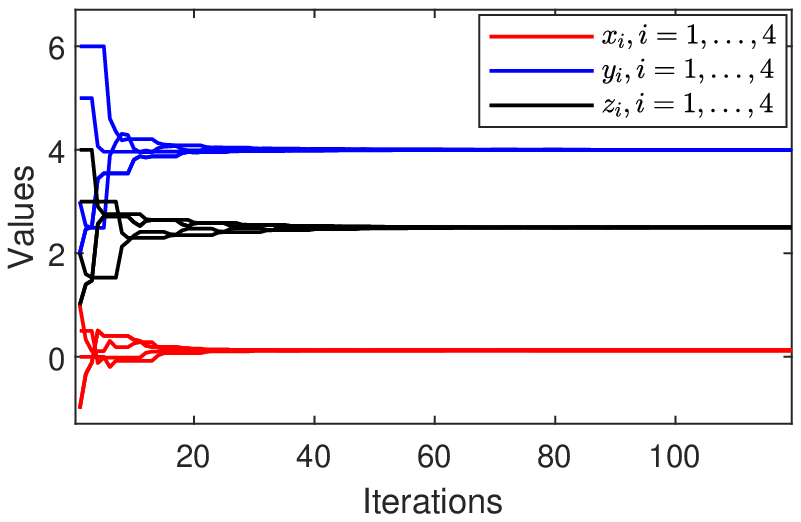}
\caption{Evolutions of two components of the agents's vectors.}
\label{fig:consensus_switching_coordinates}
\end{subfigure}
\caption{Consensus of four agents under \eqref{eq:consensus_switching_net}.}
\end{figure}
Consider a system of four agents whose state vectors are defined in $\mb{R}^2$. The graphs of the system $\mc{G}_{\sigma},\sigma=1,2,3,4$ are illustrated in Fig. \ref{fig:sim_switching_graphs}, whose switching signal $\sigma(k),k\in \mb{Z}^+$ is given as follows.
\begin{equation}\label{eq:switching_signal}
\sigma(k_t) = \small\begin{cases}
1 &\mbox{if $k=8t$ or $8t+1$}\\
2 &\mbox{if $k=8t+2$ or $8t+3$}\\
3 &\mbox{if $k=8t+4$ or $8t+5$}\\
4 &\mbox{if $k=8t+6$ or $8(t+1)-1$}
\end{cases}.
\end{equation}
Note that $\mc{G}_{\sigma(k)}$ is jointly connected in every time interval $[k_t,k_{t+1}-1]=[8t,8t+7],t\in \mb{Z}^+$, while there is only one positive definite/semi-definite edge in each graph $\mc{G}_{\sigma(k)},\sigma(k)\in \{1,2,3,4\}$. The matrix-weights of the system are given as:
\begin{align*}
\m{A}_{12}(\mc{G}_1)&=\small\begin{bmatrix}
1 &0 &0\\
0 &1.2 &0.2\\
0 &0.2 &1 
\end{bmatrix},
\m{A}_{14}(\mc{G}_2)=\begin{bmatrix}
1 &0.5 &0\\
0.5 &1 &0\\
0 &0 &1.3
\end{bmatrix},\\
\m{A}_{23}(\mc{G}_3)&=\small\begin{bmatrix}
1 &0.2 &0\\
0.2 &1.2 &0\\
0 &0 &0
\end{bmatrix},
\m{A}_{23}(\mc{G}_4)=\begin{bmatrix}
0 &0 &0\\
0 &1 &0.2\\
0 &0.2 &1.2
\end{bmatrix},
\end{align*}
and are zero matrices otherwise. Note that $\m{A}_{23}(\mc{G}_3)$ and $\m{A}_{23}(\mc{G}_4)$ are positive semidefinite, while it can be verified that $\m{A}_{23}(\mc{G}_3)+\m{A}_{23}(\mc{G}_4)>0$.

The initial vectors of the agents are given as: $\m{x}_1(0)=[-1,2,1]^\top,\m{x}_2(0)=[1,3,2]^\top,$ $\m{x}_3(0)=[0,6,3]^\top,$ and $\m{x}_4(0)=[0.5,5,4]^\top$. It observed in Fig. \ref{fig:consensus_switching_coordinates} that the agents achieve a consensus as coordinates of $\m{x}_i$, say $x_i,y_i$ and $z_i,i=1,2,3,4$ converge to the same values, respectively.
\subsection{Consensus of multi-agent systems with asymmetric and positive-semidefinite matrix weights}
Consider a system of five agents whose state vectors are defined in $3$D and interaction graph is connected. We associate each agent $i$ with a state vector $\m{p}_i\in \mb{R}^3$. In addition, agent $i$ can measure the relative vectors $(\m{p}_i-\m{p}_j)$ to some neighboring agents $j$. The initial vectors of the agents are given as $\m{p}_1(0)=[-2,-2,4]^\top,\m{p}_2(0)=[1,-3,2]^\top,$ $\m{p}_3(0)=[0,7,0]^\top,\m{p}_4(0)=[5,1,0]^\top,$ and $\m{p}_5(0)=[-1,5,0]^\top$. The matrix weights of the agents are given as follows:
\begin{align*}
\m{A}_1&=\m{A}_2=\small\begin{bmatrix}
0.6518   &-0.2604   &-0.3914\\
   -0.2604    &0.3086   &-0.0482\\
   -0.3914   &-0.0482    &0.4396
\end{bmatrix},\\
\m{A}_3&=\m{A}_4=\m{A}_5=\small\begin{bmatrix}
0.4         &0   &0\\
0   & 1  &0\\
0   &0  & 0
\end{bmatrix}.
\end{align*}
Such (positive semi-definite) matrix weights are chosen such that $\m{p}_1$ and $\m{p}_2$ lie in the plane $\mc{X}_1: x+y+z = 0$, while the evolutions of $\m{p}_3,\m{p}_4$ and $\m{p}_5$ are constrained in the plane $\mc{X}_2: z = 0$, as illustrated in Fig. \ref{fig:consensus_of_robots}.  

The evolutions of the agents' state vectors generated by \eqref{eq:constrained_consensus} are depicted in Fig. \ref{fig:consensus_of_robots}. It is observed that the state vectors converge to a consensus in the set $\mc{X}_1\cap\mc{X}_2$.
\begin{figure}[t]
\centering
\begin{subfigure}[b]{0.45\textwidth}
\centering
\includegraphics[height=3.5cm]{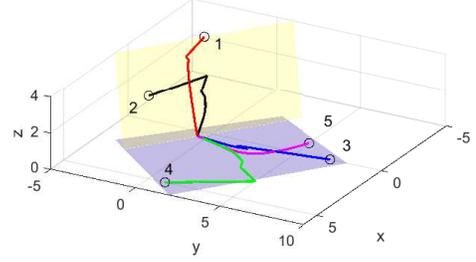}
\caption{Evolutions of the agent vectors (solid lines).}
\end{subfigure}\\
\begin{subfigure}[b]{0.45\textwidth}
\centering
\includegraphics[height=4cm]{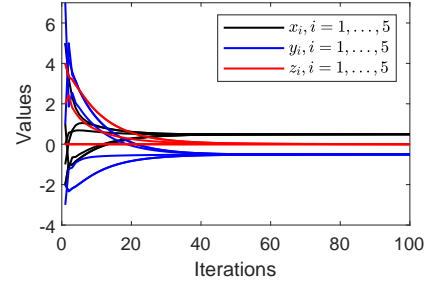}
\caption{Evolutions of the coordinates of the agents.}
\end{subfigure}
\caption{Consensus control of five agents in $\mb{R}^3$ under consensus law \eqref{eq:constrained_consensus}. State vectors of agents $\{1,2\}$ and $\{3,4,5\}$ lie in two distinct planes.}
\label{fig:consensus_of_robots}
\end{figure}
\section{Conclusion}\label{sec:conclusion}
In this paper, we investigated discrete-time matrix-weighted consensus schemes for multi-agent systems over undirected and connected graphs under various scenarios. When the network has symmetric matrix-weights, we showed that a consensus is achieved if the agents' update rates are sufficiently small and the interaction graph has a positive spanning tree. When the network graph is time-varying, joint connectedness condition of the network graph is sufficient for the agents to reach a consensus.
In a special case of consensus with non-symmetric matrix weights, under some certain conditions, the agents are shown to a achieve a consensus.

An application of the discrete-time matrix-weighted consensus to distributed optimization and machine learning is left as future work.
\appendix
\renewcommand{\thesection}{A\arabic{section}}
%
\subsection{Proof of Lemma \ref{lm:sprectral_radius_unconstrained}}\label{app:sprectral_radius_unconstrained}
i) We first show that $2\m{G}-\m{L}>0$. Indeed, for an arbitrary nonzero vector $\m{y}=[\m{y}_1^\top,\ldots,\m{y}_n^\top]^\top\in \mb{R}^{nd}$, we have
\begin{align*}
&\m{y}^\top(2\m{G}-\m{L})\m{y}=\m{y}^\top(\m{D}+\m{A})\m{y}\\
&\qquad+2\m{y}^\top\mathrm{diag}\Big(\big\{(||\m{D}_i||+\beta_i)\m{I}_d-\m{D}_i\big\}_{i=1}^n\Big)\m{y}\\
&=\textstyle\sum_{(i,j)\in \mc{E}}(\m{y}_i+\m{y}_j)^\top\m{A}_{ij}(\m{y}_i+\m{y}_j)\\
&\qquad+2\m{y}^\top\mathrm{diag}\Big(\big\{(||\m{D}_i||+\beta_i)\m{I}_d-\m{D}_i\big\}_{i=1}^n\Big)\m{y}> 0.
\end{align*}
Since $\m{G}$ is diagonal and positive definite we can write $\m{G}=\m{G}^\frac{1}{2}\m{G}^\frac{1}{2}$ with $\m{G}^\frac{1}{2}$ is also a positive definite matrix. Multiplying $\m{G}^{-\frac{1}{2}}$ on both sides of $2\m{G}-\m{L}>0$ yields
\begin{equation*}
2\m{I}_{dn}-\m{G}^{-\frac{1}{2}}\m{L}\m{G}^{-\frac{1}{2}}>0.
\end{equation*}
In addition, it is noted that $\m{G}^{-\frac{1}{2}}\m{L}\m{G}^{-\frac{1}{2}}\geq 0$ due to the positive definiteness of $\m{G}^{-\frac{1}{2}}$ and the positive semidefiniteness of $\m{L}$. Since the matrices $\m{G}^{-\frac{1}{2}}\m{L}\m{G}^{-\frac{1}{2}}$ and $\m{G}^{-1}\m{L}$ are similar, i.e., $\m{G}^{-\frac{1}{2}}\m{L}\m{G}^{-\frac{1}{2}}=\m{G}^{\frac{1}{2}}(\m{G}^{-1}\m{L})\m{G}^{-\frac{1}{2}}$, the two matrices share the same spectrum. It follows that $\lambda(\m{G}^{-1}\m{L})\in [0,2)$. Consequently, $-1<\lambda(\m{I}_{dn}-\m{G}^{-1}\m{L})\leq 1$ and hence $\rho(\m{I}_{dn}-\m{G}^{-1}\m{L})=1$. The eigenvectors correspond to the unity eigenvalues of $(\m{I}_{dn}-\m{G}^{-1}\m{L})$ are $\m{v}\in \text{null}(\m{L})$.

We show ii) as follows. It follows from i) that the eigenvectors $\m{v}_i\in \mb{R}^{dn},i=1,\ldots l_1, d\leq l_1<dn$ corresponding to the unity eigenvalues of $(\m{I}_{dn}-\m{G}^{-1}\m{L})$ are the eigenvectors of $\m{L}$ corresponding to the zero eigenvalues of $\m{L}$. Since the Laplacian $\m{L}$ is real symmetric its eigenvectors are linearly independent, or equivalently $\{\m{v}_i\}_{i=1}^{l_1}$ are linearly independent. As a result, the unity eigenvalue $1$ of $(\m{I}_{dn}-\m{G}^{-1}\m{L})$ is \textit{semisimple} as its geometric and algebraic multiplicity are equal. This shows ii). 

\subsection{Proof of Theorem \ref{thm:unconstrained_convergence}}\label{app:unconstrained_convergence}
It follows from Eqs. \eqref{eq:unconstrained_update_matrix_form} and \eqref{eq:infinity_matrix} we have that
\begin{align*}
&\lim_{k\rightarrow \infty}\m{x}(k)=(\m{I}_{dn}-\m{G}^{-1}\m{L})^\infty\m{x}(0)\\
&=(\m{1}_n\otimes \m{I}_d)[\m{u}_1,\ldots,\m{u}_d]^\top\m{x}(0)+\textstyle\sum_{i=d+1}^{l_1}\big(\m{u}_i^\top\m{x}(0)\big)\m{v}_i\\
&=(\m{1}_n\otimes \m{I}_d)\hat{\m{x}}+\textstyle\sum_{i=d+1}^{l_1}\big(\m{u}_i^\top\m{x}(0)\big)\m{v}_i,
\end{align*}
where $\hat{\m{x}}:=[\m{u}_1,\ldots,\m{u}_d]^\top\m{x}(0)\in \mb{R}^{d}$. It is noted that $\m{v}_i \not\in \text{range}(\m{1}_n\otimes \m{I}_d),\forall i=d+1,\ldots, l_1$, and such an initial vector $\m{x}(0)\perp \text{range}\{\m{u}_i\}_{d+1}^{l_1}$ is contained in a zero measure set. It then follows from the preceding relation that $\m{x}(k)\rightarrow (\m{1}_n\otimes \m{I}_d)\hat{\m{x}}$ as $k\rightarrow \infty$, for an arbitrary initial vector $\m{x}(0)\in \mb{R}^{dn}$, if and only if $\text{null}(\m{L})=\m{1}_n\otimes \m{I}_d$.

We show the geometric convergence of $\m{x}(k)\rightarrow (\m{1}_n\otimes \hat{\m{x}})$ as follows.
\begin{align*}
&||\m{x}(k)-(\m{1}_n\otimes \hat{\m{x}})||=\\
&=||\left((\m{I}_{dn}-\m{G}^{-1}\m{L})^k-(\m{1}_n\otimes \m{I}_d)[\m{u}_1,\ldots,\m{u}_d]\right)\m{x}(0)||\\
&=||(\m{V}\m{J}^k\m{V}^{-1}-\m{V}\m{J}^\infty\m{V}^{-1})\m{x}(0)||\\
&\leq||\mathrm{diag}(\m{0},\m{J}_{l_2}^k,\ldots,\m{J}_{l_p}^k)|| ||\m{x}(0)||\\
&\leq |\lambda_{d+1}|^k||\m{x}(0)||,
\end{align*}
where $|\lambda_{d+1}|<1$ is the second largest eigenvalue in magnitude of $(\m{I}_{dn}-\m{G}^{-1}\m{L})$. This completes the proof.

\subsection{Proof of Theorem \ref{thm:switching_graph}}\label{app:switching_graph}
 It is first noted that $(\m{1}_n^\top\otimes \m{I}_d)\m{x}(k+1)=(\m{1}_n^\top\otimes \m{I}_d)\m{x}(0)$ is invariant with respect to \eqref{eq:consensus_switching_net_matrix} and so is the network centroid $\bar{\m{x}}=(\m{1}_n^\top/n\otimes \m{I}_d)\m{x}(k)$. Let $\tilde{\m{x}}_i(k)=\m{x}_i(k)-\bar{\m{x}}$ and $\tilde{\m{x}}(k)=[\tilde{\m{x}}_1^\top(k),\ldots,\tilde{\m{x}}_n^\top(k)]^\top$. Then, we can rewrite \eqref{eq:consensus_switching_net_matrix} as 
\begin{equation}\label{eq:switching_net_err_equation}
\tilde{\m{x}}(k+1) = \tilde{\m{x}}(k)-\alpha\m{L}_{\sigma(k)}\tilde{\m{x}}(k).
\end{equation}
Consider the Lyapunov function $V(\tilde{\m{x}}(k))=\tilde{\m{x}}(k)^\top\tilde{\m{x}}(k)$, which is positive definite and radially unbounded. Furthermore, w.r.t. \eqref{eq:consensus_switching_net_matrix} one has
\begin{align*}
&V(\tilde{\m{x}}(k+1)) - V(\tilde{\m{x}}(k))\\
&=\tilde{\m{x}}(k)^\top(\m{I}_{dn}-\alpha\m{L}_{\sigma(k)})^\top(\m{I}_{dn}-\alpha\m{L}_{\sigma(k)})\tilde{\m{x}}(k)\\
&\quad-\tilde{\m{x}}(k)^\top\tilde{\m{x}}(k)\\
&=-\alpha\tilde{\m{x}}(k)^\top(2\m{L}_{\sigma(k)}-\alpha\m{L}_{\sigma(k)}^2)\tilde{\m{x}}(k)\\
&\leq -(\mu^{-1}-\alpha)\tilde{\m{x}}(k)^\top\m{L}_{\sigma(k)}^2\tilde{\m{x}}(k)\\
&=-(\mu^{-1}-\alpha)||\m{L}_{\sigma(k)}\tilde{\m{x}}(k)||^2\leq 0, \numberthis \label{eq:switching_graph_nonincreasingV}
\end{align*}
where the first inequality follows from the fact that
$\m{L}_{\sigma(k)}-(1/\mu)\m{L}_{\sigma(k)}^2 \geq 0$ with $\mu = \max_{\sigma(k)}||\m{L}_{\sigma(k)}||$, and in the last inequality we use the condition $\alpha < 1/\mu$. It follows that $V(\tilde{\m{x}}(k+1))$ is non-increasing w.r.t. \eqref{eq:consensus_switching_net_matrix}  and hence $\{\m{x}(k)\}$ is bounded. In addition, $\lim_{k\rightarrow\infty}V(\tilde{\m{x}}(k))=\sum_{i=1}^k\Big(V(\tilde{\m{x}}(i))-V(\tilde{\m{x}}(i-1))\Big) + V(\tilde{\m{x}}(0))$ exists. This further implies that the sequence $\{V(\tilde{\m{x}}(k+1)) - V(\tilde{\m{x}}(k))\}$ is \textit{summable} and consequently, $\lim_{k\rightarrow\infty}V(\tilde{\m{x}}(k+1)) - V(\tilde{\m{x}}(k))=0$. Thus, by \eqref{eq:switching_graph_nonincreasingV}, we have
\begin{equation}\label{eq:switching_graph_steady_state_k}
\lim_{k\rightarrow\infty}\m{L}_{\sigma(k)}\tilde{\m{x}}(k)=\m{0}.
\end{equation}
From this relation, we next show that the following relation holds for all $s\in \mb{Z}^+$
\begin{equation}\label{eq:switching_graph_steady_state_k+1}
\lim_{k\rightarrow\infty}\m{L}_{\sigma(k+s)}\tilde{\m{x}}(k)=\m{0}.
\end{equation}
To proceed, using the relation $\tilde{\m{x}}(k)=\tilde{\m{x}}(k+1)+\alpha\m{L}_{\sigma(k)}\tilde{\m{x}}(k)$ (due to \eqref{eq:switching_net_err_equation}), one has 
\begin{align*}
&\m{L}_{\sigma(k+s)}\tilde{\m{x}}(k)=\m{L}_{\sigma(k+s)}(\tilde{\m{x}}(k+1)+\alpha\m{L}_{\sigma(k)}\tilde{\m{x}}(k))\\
&=\m{L}_{\sigma(k+s)}(\tilde{\m{x}}(k+2)+\alpha\m{L}_{\sigma(k+1)}\tilde{\m{x}}(k+1)+\alpha\m{L}_{\sigma(k)}\tilde{\m{x}}(k))\\
&=\m{L}_{\sigma(k+s)}\Big(\tilde{\m{x}}(k+s)+\alpha\m{L}_{\sigma(k+s-1)}\tilde{\m{x}}(k+s-1)+\ldots\\
&\quad+\alpha\m{L}_{\sigma(k+1)}\tilde{\m{x}}(k+1)+\alpha\m{L}_{\sigma(k)}\tilde{\m{x}}(k)\Big)
\end{align*}
Therefore, \eqref{eq:switching_graph_steady_state_k+1} follows from the fact that $\lim_{k\rightarrow \infty}\m{L}_{\sigma(k+s)}\tilde{\m{x}}(k+s)=\lim_{k\rightarrow \infty}\m{L}_{\sigma(k)}\tilde{\m{x}}(k)=\m{0}$ for all $s\in \mb{Z}^+$. Moreover, it is follows from \eqref{eq:switching_graph_steady_state_k+1} that 
$$\lim_{k_t\rightarrow\infty}\m{L}_{\sigma(k_t+s)}\tilde{\m{x}}(k_t)=\m{0},~\forall s\in \mb{Z}^+.$$
By summing the preceding relations over $s$ from $0$ to $(k_{t+1}-k_t-1)$, we have
\begin{equation}
\lim_{k_t\rightarrow\infty}\sum_{k=k_t}^{k_{t+1}-1}\m{L}_{\sigma(k)}\tilde{\m{x}}(k_t)=\m{0}.
\end{equation}
Since $\tilde{\m{x}}(k_t)\perp\text{null}(\sum_{k=k_t}^{k_{t+1}-1}\m{L}_{\sigma(k)})=\text{range}(\m{1}_n\otimes \m{I}_d)$ due to the joint connectedness condition in Assumption \ref{ass:jointly_connected}, we have $\lim_{k_t\rightarrow\infty}\tilde{\m{x}}(k_t)=\m{0}$. This completes the proof.

\subsection{Proof of Lemma \ref{lm:asym_non_increasing_function}}\label{app:asym_non_increasing_function}
First, it follows from \eqref{eq:asym_consensus_matrix_form} and \eqref{eq:asym_feasible_direction} we have that
\begin{align*}
&(\m{x}(k+1)-\m{x}(k))^\top\m{L}^\mathrm{o}\m{x}(k)\\
&=-(\m{x}(k+1)-\m{x}(k))^\top\m{G}^{-1}(\m{x}(k+1)-\m{x}(k))\\
&= -\sum_{i=1}^n\frac{1}{\alpha_i}(\m{x}_i(k+1)-\m{x}_i(k))^\top\m{A}_i^{-1}(\m{x}_i(k+1)-\m{x}_i(k))\\
&\leq-\gamma_{\min}\alpha_{\max}^{-1}||\m{x}(k+1)-\m{x}(k)||^2.
\end{align*}
Then, since $\nabla V$ is Lipschitz with constant $L_V$, we have \cite{Bertsekas1989}
\begin{align*}
&V(\m{x}(k+1))-V(\m{x}(k))\leq (\m{x}(k+1)-\m{x}(k))^\top\nabla V(\m{x}(k))\\
&\qquad+(L_V/2)||\m{x}(k+1)-\m{x}(k)||^2\\
&= (\m{x}(k+1)-\m{x}(k))^\top\m{L}^\mathrm{o}\m{x}(k)+\frac{L_V}{2}||\m{x}(k+1)-\m{x}(k)||^2\\
&\leq- (\gamma_{\min}\alpha_{\max}^{-1}-L_V/2)||\m{x}(k+1)-\m{x}(k)||^2\leq 0,
\end{align*}
if $0<\alpha_{\max}<2\gamma_{\min}/L_V$.

\subsection{Proof of Theorem \ref{thm:asym_pos_weight_consensus}}\label{app:asym_pos_weight_consensus}
It follows from the non-increase of $V(\m{x}(k))=\sum_{(i,j)\in \mc{E}}||\m{x}_i(k)-\m{x}_j(k)||^2$ that $\max_{(i,j)\in \mc{E}}||\m{x}_i(k)-\m{x}_j(k)||$ is bounded. Moreover, from \eqref{eq:asym_consensus_matrix_form}, one has $(\m{1}_n\otimes \m{I}_d)\m{G}^{-1}\m{x}(k+1)=(\m{1}_n\otimes \m{I}_d)\m{G}^{-1}\m{x}(k)-(\m{1}_n\otimes \m{I}_d)\m{G}^{-1}\m{G}\m{L}^\mathrm{o}\m{x}(k)=(\m{1}_n\otimes \m{I}_d)\m{G}^{-1}\m{x}(k)$, which implies that $\sum_{i=1}^n(\alpha_i\m{A}_i)^{-1}\m{x}_i(k)$ is invariant. As a result, $\{\m{x}(k)\}$ is bounded, and hence there exists a convergent subsequence $\{\m{x}(k_l)\},l \in \mb{Z}^+$ and a limit point $\m{x}^{\dag}\in \mb{R}^{dn}$ such that $\lim_{l\rightarrow \infty}\m{x}(k_l)=\m{x}^\dag$.

By summing up the inequalities in \eqref{eq:asym_non_increasing_function} over $k$ from $0$ to $\infty$, we have
\begin{align*}
\sum_{k=0}^{\infty}||\m{x}(k+1)-\m{x}(k)||^2&\leq \gamma_{\min}^{-1}\alpha_{\max} \big(V(\m{x}(0))-V(\m{x}(\infty))\big)\\
&\leq \gamma_{\min}^{-1}\alpha_{\max}V(\m{x}(0)).
\end{align*}
 It follows that $(\m{x}(k+1)-\m{x}(k))$ is a \textit{square-summable} sequence and hence $||\m{x}(k+1)-\m{x}(k)||\rightarrow \m{0}$ as $k\rightarrow \infty$. Therefore, from \eqref{eq:asym_consensus_matrix_form}, $||\m{L}^\mathrm{o}\m{x}(k)||\leq||\m{G}^{-1}||||\m{x}(k+1)-\m{x}(k)||\rightarrow 0$ as $k\rightarrow \infty$. As a result, $\m{x}^\dag \in \text{null}(\m{L}^\mathrm{o})$ and hence $\m{x}^\dag=\m{1}_n\otimes \m{x}^*$ for a point $\m{x}^*\in \mb{R}^d$.

Since $\sum_{i=1}^n(\alpha_i\m{A}_i)^{-1}\m{x}_i(k)$ is invariant, we have $(\sum_{i=1}^n\alpha_i^{-1}\m{A}_i^{-1})\m{x}^*=\sum_{i=1}^n\alpha_i^{-1}\m{A}_i^{-1}\m{x}_i(0)\Leftrightarrow \m{x}^*=(\sum_{i=1}^n\alpha_i^{-1}\m{A}_i^{-1})^{-1}\sum_{i=1}^n\alpha_i^{-1}\m{A}_i^{-1}\m{x}_i(0)$, which is a fixed point. It follows that every sequence $\{\m{x}(k),k\in \mb{Z}^+\}$ converges to $\m{1}_n\otimes \m{x}^*$. 

\subsection{Proof of Theorem \ref{thm:constrained_consensus}}\label{app:constrained_consensus}
We firstly show the boundedness of the sequence $\{\m{x}(k)\}$ generated by \eqref{eq:constrained_consensus} and then prove its convergence to a consensus.
\subsubsection{Boundedness evolution}
It follows from the non-increase of $V(\m{x}(k))=\sum_{(i,j)\in \mc{E}}||\m{x}_i(k)-\m{x}_j(k)||^2$ that $\max_{(i,j)\in \mc{E}}||\m{x}_i(k)-\m{x}_j(k)||$ is bounded. Moreover, for an arbitrary point $\m{x}^\prime\in \mc{X}$, we can rewrite \eqref{eq:constrained_consensus} as
\begin{align*}
&(\m{x}_i(k+1)-\m{x}^\prime)=(\m{x}_i(k)-\m{x}^\prime)\\
&-\alpha_i\textstyle\sum_{j\in \mc{N}_i}\m{A}_{i}\Big((\m{x}_i(k)-\m{x}^\prime)-(\m{x}_j(k)-\m{x}^\prime)\Big),\numberthis\label{eq:constrained_consensus_relative_state}
\end{align*}
for all $i\in \mc{V}$. Left-multiplying $\alpha_i^{-1}\m{A}_{i}^\dag$ on both sides of \eqref{eq:constrained_consensus_relative_state} yields 
\begin{align*}
&\alpha_i^{-1}\m{A}_{i}^\dag(\m{x}_i(k+1)-\m{x}^\prime)=\alpha_i^{-1}\m{A}_{i}^\dag(\m{x}_i(k)-\m{x}^\prime)\\
&-\textstyle\sum_{j\in \mc{N}_i}\m{P}_{\mc{TX}_i}\Big((\m{x}_i(k)-\m{x}^\prime)-(\m{x}_j(k)-\m{x}^\prime)\Big).\numberthis\label{eq:relative_state_project_to_Xi}
\end{align*}
Consider any nonzero vector $\m{v}_i=\m{v}_i^t+\m{v}_i^n$, where $\m{v}_i^t=\m{P}_{\mc{TX}_i}\m{v}_i$ and $\m{v}_i^n=(\m{I}_d-\m{A}_i^\dag\m{A}_i)\m{v}_i$ (see e.g. Fig. \ref{fig:asym_pos_semidef_weight}). Let $\m{P}_{\mc{TX}}\in \mb{R}^{d\times d}$ be the projection matrix that projects any vector onto the tangent space $\mc{TX}$. Note that when $\mc{X}$ is a singleton, $\m{P}_{\mc{TX}}=\m{0}$. Then, for every $i\in \mc{V}$, we have
\begin{align*}
\m{P}_{\mc{TX}}\m{v}_i&=\m{P}_{\mc{TX}}\m{v}_i^t\\
\Leftrightarrow \m{P}_{\mc{TX}}\m{v}_i&=\m{P}_{\mc{TX}}\m{P}_{\mc{TX}_i}\m{v}_i. \numberthis
\end{align*}
Using the preceding relation and by left-multiplying $\m{P}_{\mc{TX}}$ on both sides of \eqref{eq:relative_state_project_to_Xi}, for all $i\in \mc{V}$ we obtain
\begin{align*}
&\alpha_i^{-1}\m{P}_{\mc{TX}}\m{A}_{i}^\dag(\m{x}_i(k+1)-\m{x}^\prime)=\alpha_i^{-1}\m{P}_{\mc{TX}}\m{A}_{i}^\dag(\m{x}_i(k)-\m{x}^\prime)\\
&-\m{P}_{\mc{TX}}\textstyle\sum_{j\in \mc{N}_i}\Big((\m{x}_i(k)-\m{x}^\prime)-(\m{x}_j(k)-\m{x}^\prime)\Big).\numberthis\label{eq:}
\end{align*}
By adding the preceding equations over $i$ from $1$ to $n$, one has $\m{P}_{\mc{TX}}\sum_{i=1}^n\alpha_i^{-1}\m{A}_{i}^\dag(\m{x}_i(k+1)-\m{x}^\prime)=\m{P}_{\mc{TX}}\sum_{i=1}^n\alpha_i^{-1}\m{A}_{i}^\dag(\m{x}_i(k)-\m{x}^\prime)$ $\Leftrightarrow \m{P}_{\mc{TX}}\sum_{i=1}^n\alpha_i^{-1}\m{A}_{i}^\dag(\m{x}_i(k+1)-\m{x}^\prime)=\m{P}_{\mc{TX}}\sum_{i=1}^n\alpha_i^{-1}\m{A}_{i}^\dag(\m{x}_i(0)-\m{x}^\prime)$, which is invariant for all $k\in \mb{Z}^+$. Consequently, the sequence $\{\m{x}(k)\}$ generated by \eqref{eq:constrained_consensus} is bounded.

\subsubsection{Convergence to a consensus} By summing up the inequalities in \eqref{eq:constrained_non_increasing_function} over $k$ from $0$ to $\infty$, we have
$
\sum_{k=0}^{\infty}||\m{x}(k+1)-\m{x}(k)||^2\leq \gamma_{\min}^{-1}\alpha_{\max}\big(V(\m{x}(0))-V(\m{x}(\infty))\big)\leq \gamma_{\min}^{-1}\alpha_{\max}V(\m{x}(0)).
$
 It follows that $||\m{x}(k+1)-\m{x}(k)||\rightarrow \m{0}$ or equivalently $\m{x}(k)\rightarrow \hat{\m{x}}:=[\hat{\m{x}}_1^\top,\ldots,\hat{\m{x}}_n^\top]^\top\in \mb{R}^{dn}$, as $k\rightarrow \infty$. Furthermore, from \eqref{eq:relative_state_project_to_Xi}, for an arbitrary point $\m{x}^\prime\in \mc{X}$, we have that \begin{align*}
 &\m{P}_{\mc{TX}_i}\textstyle\sum_{j\in \mc{N}_i}\big((\hat{\m{x}}_i-\m{x}^\prime)-(\hat{\m{x}}_j-\m{x}^\prime)\big)= \m{0},~\forall i\in \mc{V}
\\
\Leftrightarrow &|\mc{N}_i|(\hat{\m{x}}_i-\m{x}^\prime)=\textstyle\sum_{j\in \mc{N}_i}\m{P}_{\mc{TX}_i}(\hat{\m{x}}_j-\m{x}^\prime),~\forall i\in \mc{V}, \numberthis \label{eq:constrained_relation_at_limit}
\end{align*} 
where the last equality follows from $\m{P}_{\mc{TX}_i}(\hat{\m{x}}_i-\m{x}^\prime)=(\hat{\m{x}}_i-\m{x}^\prime),\forall i\in \mc{V}$.

We define the index set $\mc{I}:=\{i\in \mc{V}:i=\mathrm{argmax}_{i\in \mc{V}}||\hat{\m{x}}_i-\m{x}^\prime||\}$. Then, consider an agent $i\in \mc{I}$, we have 
\begin{align*}
\big\lVert\textstyle\sum_{j\in \mc{N}_i}\m{P}_{\mc{TX}_i}(\hat{\m{x}}_j-\m{x}^\prime)\big\rVert 
&\leq \textstyle\sum_{j\in \mc{N}_i}||\m{P}_{\mc{TX}_i}(\hat{\m{x}}_j-\m{x}^\prime)||\\
&\leq\textstyle\sum_{j\in \mc{N}_i}||\hat{\m{x}}_j-\m{x}^\prime|| \\
&\leq  |\mc{N}_i|||\hat{\m{x}}_i-\m{x}^\prime||,
\end{align*}
where the equality holds only if $\hat{\m{x}}_j\in \mc{X}_i$ and $||\hat{\m{x}}_j-\m{x}^\prime||=||\hat{\m{x}}_i-\m{x}^\prime||$, for all $j\in \mc{N}_i$. This combines with \eqref{eq:constrained_relation_at_limit} lead to $\hat{\m{x}}_j\equiv  \hat{\m{x}}_i,\forall j\in \mc{N}_i$, and consequently, $j\in \mc{I},\forall j \in \mc{N}_i$. By repeating the above argument for all agents $j\in \mc{I}$ until all the agents in the system have been visited (due to the connectedness of the graph $\mc{G}$), we obtain $\hat{\m{x}}_i\equiv \m{x}^*\in \mc{X},\forall i\in \mc{V}$.

The remainder of the proof is amount to finding an explicit expression for the consensus point $\m{x}^*$. Let $\bar{\m{A}}:=\sum_{i=1}^n\alpha_i^{-1}\m{A}_{i}^\dag\geq 0$. For an arbitrary point $\m{x}^\prime\in \mc{X}$, we have 
\begin{align*}
\m{P}_{\mc{TX}}\bar{\m{A}}(\m{x}^*-\m{x}^\prime)&=\m{P}_{\mc{TX}}\textstyle\sum_{i=1}^n\alpha_i^{-1}\m{A}_{i}^\dag(\m{x}_i(0)-\m{x}^\prime)\\
\Leftrightarrow \m{P}_{\mc{TX}}\bar{\m{A}}\m{P}_{\mc{TX}}(\m{x}^*-\m{x}^\prime)&=\m{P}_{\mc{TX}}\textstyle\sum_{i=1}^n\alpha_i^{-1}\m{A}_{i}^\dag(\m{x}_i(0)-\m{x}^\prime)
\end{align*} 
where we use the relation $\m{P}_{\mc{TX}}(\m{x}^*-\m{x}^\prime)=(\m{x}^*-\m{x}^\prime)$.
It is noted that $\mathrm{range}(\m{P}_{\mc{TX}})\subseteq \mathrm{range}(\bar{\m{A}})$ and hence $(\m{x}^*-\m{x}^\prime)\in\mathrm{range}(\m{P}_{\mc{TX}})= \mathrm{range}(\m{P}_{\mc{TX}}\bar{\m{A}}\m{P}_{\mc{TX}})$. Therefore, the consensus point $\m{x}^*$ is uniquely defined as
$$\m{x}^*=\m{x}^\prime+(\m{P}_{\mc{TX}}\bar{\m{A}}\m{P}_{\mc{TX}})^\dag\m{P}_{\mc{TX}}\sum_{i=1}^n\alpha_i^{-1}\m{A}_{i}^\dag(\m{x}_i(0)-\m{x}^\prime),$$
where $(\m{P}_{\mc{TX}}\bar{\m{A}}\m{P}_{\mc{TX}})^\dag$ is the Moore-Penrose inverse of $\m{P}_{\mc{TX}}\bar{\m{A}}\m{P}_{\mc{TX}}$.




\nocite{}
\bibliographystyle{IEEEtran}
\bibliography{quoc2018,quoc2019}

\end{document}